\documentclass[11pt]{article}
\usepackage{amsmath, amssymb, amsthm, verbatim,enumerate,bbm, mathtools}
\usepackage{indentfirst}

\date{\today}
\allowdisplaybreaks

\theoremstyle{plain}
\newtheorem{theorem}{Theorem}[section]
\newtheorem{lemma}[theorem]{Lemma}

\newtheorem{proposition}[theorem]{Proposition}
\newtheorem{observation}[theorem]{Observation}

\title{Singularity of random symmetric matrices -- simple proof}
\author{Asaf Ferber \thanks{Department of Mathematics, University of California, Irvine. Email: {\tt asaff@uci.edu}. Research is partially supported by an NSF grant DMS-1954395.}}
\date{\today}

\begin{document}

\maketitle

\begin{abstract}
    In this paper we give a simple, short, and self-contained proof for a non-trivial upper bound on the probability that a random $\pm 1$ symmetric matrix is singular.   
\end{abstract}

\section{Introduction}

A widely studied model of discrete random matrices is that of random \emph{symmetric} $\pm1$ matrices. That is, let $M_n$ denote an $n\times n$ symmetric $\pm1$ matrix chosen uniformly from the set of all such matrices. 

One of the most natural problems is to estimate $$p(n)=\Pr[M_n \textrm{ is singular}].$$

In this model, even proving that $p(n)=o(1)$ (this problem was posed by Weiss in the early 1990s) is quite challenging and was only settled in 2005 by Costello, Tao, and Vu \cite{costello2006random}, who showed that $$p(n) = O\left(n^{-1/8 + o(1)}\right).$$ 

In their work, they introduced and studied a quadratic variant of the Erd\H{o}s-Littlewood-Offord inequality and a useful decoupling lemma, which serve as key tools in all subsequent works on this problem. 

Following some intermediate works by Nguyen \cite{nguyen2012inverse}, Vershynin \cite{vershynin2014invertibility}, and Ferber and Jain \cite{ferber2019singularity}, the current best bound on $p(n)$ is 
$$p(n)=2^{-\omega(\sqrt{n})}$$
due to Campos, Mattos, Morris, and Morrison \cite{campos2019singularity}. Moreover, as was noted in \cite{campos2019singularity}, this bound is the best one can hope to obtain using the existing technique. 

The common belief is that $p(n)=(\frac{1}{2}+o(1))^n$, which, if true, is clearly best possible, as one can check by calculating the probability that $M_n$ has at least two identical rows/columns. Therefore, in order to make a further progress, it is required to come up with new ideas/techniques to tackle this problem. The aim of this note is to provide a proof for a non-trivial (but yet, quite weak) bound for $p(n)$ which completely avoids the difficulties from the previous approach (it might introduce new difficulties though). 

Our main theorem is the following: 

\begin{theorem} \label{main} There exists some $C>0$ for which
$p(n)=O(\frac{\log^Cn}{n^{1/2}})$.
\end{theorem}

We did not try to improve the bound in Theorem \ref{main} as we wanted to keep the proof short and simple. It is plausible that with some ideas from \cite{ferber2019singularity} one could significantly improve this bound, but in order to obtain an exponential bound it seems like one needs to come up with new ideas. 

The proof is based on ideas from \cite{FJLS2018, huang2018invertibility}, but the details are much simpler.  

\section{Auxiliary lemmas}

Let $q$ be some prime number, let $\boldsymbol{a}\in \mathbb{Z}_q^n$, and let $r\in \mathbb{Z}_q$. The $r$th \emph{level set of} $\boldsymbol{a}$ is 
$$L_r(\boldsymbol{a})=\{i\in [n] \mid a_i=r\}.$$  
For convenience, we will use the notation $L_{\neq r}:=[n]\setminus L_r,$ and we also set $m(\boldsymbol{a})$ to be the size of the largest level set. 

Finally, we let
$$\mathcal L:=\{\boldsymbol{a}\in \mathbb{Z}_q^n \mid m(\boldsymbol{a})\geq n-n/\log^2n\}$$
be the set of all $\boldsymbol{a}\in \mathbb{Z}_q^n$ with some level set of size larger than $n-\frac{n}{\log^2n}$.

Now we are ready to state our auxiliary lemmas. First, let us make the following simple (but yet, useful) observation:

\begin{observation} \label{obs: trivial bound}
Let $\boldsymbol{a}\in \mathbb{Z}_q^n\setminus \{\boldsymbol{0}\}$. Then, 
$$\Pr[M_n\cdot \boldsymbol{a}=\boldsymbol{0}]\leq 2^{-n}.$$
\end{observation}

\begin{proof}
Indeed, let $1\leq j\leq n$ be some coordinate for which $a_j\neq 0\mod q$, and expose all the entries of $M_n$ but the entries in the $j$th row and column. It is now straightforward to see that we obtain the desired. 
\end{proof}

The following lemma is basically the key lemma for our proof. Roughly speaking, it asserts that if $\boldsymbol{a}\notin \mathcal{L}$, then $M_n\cdot \boldsymbol{a}$ is (more or less) equally likely to be any vector from $\mathbb{Z}_q^n$.

\begin{lemma} \label{lemma:key}
Let $q$ be a prime such that $q = O(n^{1/2}/\log^{C}n)$. Let $\boldsymbol{a}\notin \mathcal{L}$, and $\boldsymbol{v}\in \mathbb{Z}_q^n$. Then, 
$$\Pr[M_n\cdot \boldsymbol{a}=\boldsymbol{v}]=\frac{1+o(1)}{q^n}.$$ 
\end{lemma}

Before proving Lemma \ref{lemma:key}, let us first state (and prove) two simple statements that will be used in the proof of the lemma. 

\begin{proposition} \label{proposition 1}
Let $\boldsymbol{a}\notin \mathcal{L}$ and let $\boldsymbol{\ell}\in \mathbb{Z}_q^n$ be a vector with support of size $s< \frac{n}{2\log^2n}$. Then, there are at least $\frac{sn}{2\log^2n}$ pairs $1\leq i<j\leq n$ for which $\ell_ia_j+\ell_ja_i\neq 0\mod q$. 
\end{proposition}

\begin{proof}
Since $\boldsymbol{a}\notin \mathcal L$, we have that $|L_{\neq 0}(\boldsymbol{a})|\geq \frac{n}{\log^2n}$ and therefore we have that $J:=L_0(\boldsymbol{\ell})\cap L_{\neq 0}(\boldsymbol{a})$ is of size at least $|L_{\neq 0}(\boldsymbol{a})|-s\geq \frac{n}{2\log^2n}$. Now, observe that for every $i\in L_{\neq 0}(\boldsymbol{\ell})$ and $j\in J$ we have that $\ell_ia_j+\ell_ja_i=\ell_ia_j\neq 0\mod q$. In particular, there are at least
$$|L_{\neq 0}(\boldsymbol{\ell})|\cdot |J|\geq \frac{sn}{2\log^2n}$$
such pairs. This completes the proof.
\end{proof}

\begin{proposition} \label{proposition 2}
Let $\boldsymbol{a}\notin \mathcal{L}$ and let $\boldsymbol{\ell}\in \mathbb{Z}_q^n$ be a vector with support of size $s\geq \frac{n}{2\log^2n}$. Then, there are at least $\min\{s^2/20, \frac{sn}{2\log^2n}\}$  pairs $1\leq i<j\leq n$ for which $\ell_ia_j+\ell_ja_i\neq 0\mod q$.
\end{proposition}

\begin{proof}
We split into two cases:

{\bf Case 1.} $|L_0(\boldsymbol{a})\cap L_{\neq 0}(\boldsymbol{\ell})|\geq s/2$. In this case, let $I:=L_0(\boldsymbol{a})\cap L_{\neq 0}(\boldsymbol{\ell})$ and $J=L_{\neq 0}(\boldsymbol{a})$. Clearly, for all $i\in I$ and $j\in J$ we have $\ell_ia_j+\ell_ja_i=\ell_ia_j\neq 0 \mod q$, and therefore, there are at least $|I|\cdot |J|\geq \frac{sn}{2\log^2n}$ such pairs. 

{\bf Case 2.} $|L_{\neq 0}(\boldsymbol{a})\cap L_{\neq 0}(\boldsymbol{\ell})|\geq s/2$. For all $r\in \mathbb{Z}_q$ we define $J_r:=L_r(\boldsymbol{a})\cap L_{\neq 0}(\boldsymbol{\ell})$, and observe that $s':=\sum_{r\neq 0}|J_r|\geq s/2$. Next, define an auxilairy graph $G$ on vertex set $V=L_{\neq 0}(\boldsymbol{a})\cap L_{\neq 0}(\boldsymbol{\ell})$, where two vertices $i,j\in V$ are connected by an edge if and only if $\ell_ia_j+\ell_ja_i=0 \mod q$. We show that $G$ is triangle free, and therefore, by Mantel's theorem we have $e(G)\leq \frac{1}{2}\cdot \binom{|V|}{2}$. In particular, it means that there are at least $\frac{1}{2}\cdot \binom{|V|}{2}\geq s^2/20$ pairs $i,j\in V$ for which $\ell_ia_j+\ell_ja_i\neq 0 \mod q$ as desired. 

To this end, let $i,j,k\in V$ be three distinct vertices. We distinguish between three cases:

{\bf Case 2.1} $i,j,k\in J_r$ for some $r\neq 0$. Observe that 
$\ell_ia_j+\ell_ja_i=r(\ell_i+\ell_j)$, and therefore, if it equals $0\mod q$, then we have $\ell_i=-\ell_j$. Now, without loss of generality we can assume that $\ell_k\neq -\ell_j$ (the case $\ell_k\neq \ell_j$ is treated similarly). Then, 
$$\ell_ka_j+\ell_ja_k=r(\ell_k+\ell_j)\neq 0\mod q.$$

{\bf Case 2.2} $i,j\in J_{r_1}$ and $k\in J_{r_2}$ for some $r_1\neq r_2$ and both are not  $0\mod q$. If $\ell_ia_j+\ell_ja_i=r_1(\ell_i+\ell_j)\neq 0 \mod q$ then we are done. Otherwise, we have that $\ell_i=-\ell_j$. Now, consider the expressions $$\ell_ia_k+\ell_ka_i=\ell_ir_2+\ell_kr_1, \textrm{ and }\ell_ja_k+\ell_ka_j=-\ell_ir_2+\ell_kr_1.$$ 
Clearly, at least one of them is not $0\mod q$.

{\bf Case 2.3} $i\in J_{r_1}$, $j\in J_{r_2}$ and $k\in J_{r_3}$ for some distinct $r_1,r_2,$ and  $r_3$, all are not $0\mod q$. Suppose that we have
$$\ell_ia_j+\ell_ja_i=\ell_ia_k+\ell_ka_i=0\mod q.$$
(if not, then we are done). 

 In particular, it means that 
 $$\ell_i=\frac{-\ell_ja_i}{a_j}=\frac{-\ell_jr_1}{r_2},$$
and that 
$$\ell_i=\frac{-\ell_kr_1}{r_3}.$$ 

These two identities yield that 
$$0\mod q=\ell_kr_2-\ell_jr_3=\ell_ka_j-\ell_ja_k,$$ 
and in particular, since $\ell_ja_k\neq 0\mod q$, we have that 
$$\ell_ka_j+\ell_ja_k\neq 0\mod q$$
as desired. This completes the proof. 
\end{proof}

Now we are ready to prove Lemma \ref{lemma:key}. 

\begin{proof}
Let $\boldsymbol{a}\notin \mathcal L$, let $\boldsymbol{v}\in \mathbb{Z}_q^n$, and let $e_q(x)=e^{\frac{2\pi i x}{q}}$. Recalling that $m_{ij} = m_{ji}$, observe that
    \begin{align*}
    \Pr[M\cdot \boldsymbol{a}=v]&=\mathbb{E}[\delta_{0}(M\cdot \boldsymbol{a}-v)]\\
    &=\frac{1}{q^n}\sum_{\boldsymbol{\ell}\in \mathbb{Z}_q^n}\mathbb{E}[e_q(\boldsymbol{\ell}^T(M_n\boldsymbol{a}-v))]\\
    &=\frac{1}{q^n}\sum_{\boldsymbol{\ell}\in \mathbb{Z}_q^n}e_q(-\boldsymbol{\ell}^Tv)\cdot \mathbb{E}[e_q(\sum_{i,j}m_{ij}\ell_ia_j)]\\
    &=\frac{1}{q^n}\sum_{\boldsymbol{\ell}\in \mathbb{Z}_q^n}e_q(-\boldsymbol{\ell}^Tv)\prod_{1\leq i<j\leq n}\mathbb{E}[e_q(m_{ij}(\ell_ia_j+\ell_ja_i))]\prod_{i=1}^n\mathbb{E}[e_q(m_{ii}\ell_ia_i)]\\
    &=\frac{1}{q^n}+\frac{1}{q^n}\sum_{\boldsymbol{\ell}\neq \boldsymbol{0}\in \mathbb{Z}_q^n}e_q(-\boldsymbol{\ell}^Tv)\prod_{1\leq i<j\leq n}\cos\left(\frac{2\pi}{q}(\ell_ia_j+\ell_ja_i)\right)\prod_{i=1}^n\cos\left(\frac{2\pi \ell_ia_i}{q}\right).
\end{align*}

This implies that 

\begin{align*}
    \left|\Pr[M\cdot \boldsymbol{a}=v]-\frac{1}{q^n}\right|&\leq \frac{1}{q^n}\sum_{\boldsymbol{\ell}\neq \boldsymbol{0}\in \mathbb{Z}_q^n}\prod_{1\leq i<j\leq n}\left|\cos\left(\frac{2\pi}{q}(\ell_ia_j+\ell_ja_i)\right)\right|.
    \end{align*}
    
Therefore, it is enough to show that 

$$Error:=\sum_{\boldsymbol{\ell}\neq \boldsymbol{0}\in \mathbb{Z}_q^n}\prod_{1\leq i<j\leq n}\left|\cos\left(\frac{2\pi}{q}(\ell_ia_j+\ell_ja_i)\right)\right|=o(1).$$

Using the following simple estimate 
$$|\cos \frac{\pi m}{q}|\leq e^{-\frac{2}{q^2}},$$
which holds for all $m\neq 0 \mod q,$ we can upper bound

$$Error\leq  \sum_{\boldsymbol{\ell}\in \mathbb{Z}_q^n\setminus \{\boldsymbol{0}\}}e^{-2\cdot N(\boldsymbol{\ell},\boldsymbol{a})/q^2},$$
where $N(\boldsymbol{\ell},\boldsymbol{a})=\left|\{ (i,j)\in [n]^2 : \ell_ia_j+\ell_ja_i\neq 0\mod q\}\right|.$

Finally, to complete the proof, we split the above sum according to the size of the support of $\boldsymbol{\ell}$, and using Propositions \ref{proposition 1} and \ref{proposition 2} we obtain that 

$$Error\leq \sum_{s=1}^{n}\binom{n}{s}q^se^{-sn/q^2\log^2n}+\sum_{s=n/\log^2n}^{n}\binom{n}{s}(q-1)^se^{-s^2/20q^2},$$
which can be easily seen to be $o(1)$ as long as $q=O(n^{1/2}/\log^Cn)$. This completes the proof. 
\end{proof}

\section{Proof of Theorem \ref{main}}

We work over $\mathbb{Z}_q$, where $q=\Theta\left(\frac{n^{1/2}}{\log^C n}\right)$ is some prime, and observe that 
$$p(n)\leq \Pr[M_n \textrm{ is singular over }\mathbb{Z}_q]:=p'(n).$$

Now, define the random variable  
$$K=|Ker_{\mathbb{Z}_q}(M_n)|$$ 
and observe that 
$$\mathbb{E}[K]=\sum_{\boldsymbol{a}\in \mathbb{Z}_q^n}\Pr[M\cdot \boldsymbol{a}=0].$$

Our goal is to show that $\mathbb{E}[K]\leq 2+o(1)$, and then, by Markov's inequality we obtain that 
$$p'(n)=\Pr[K\geq q]\leq (2+o(1))/q$$ as desired.

To this end, let us split the above according to whether $\boldsymbol{a}$ is in $\mathcal L$ or not (recall that $\mathcal L$ is the set of all vectors $\boldsymbol{a}\in \mathbb{Z}_q^n$ with a level set of size at least $n-n/\log^2n$), and by  Observation \ref{obs: trivial bound} and Lemma \ref{lemma:key} we obtain that
    \begin{align*}
        \mathbb{E}[K]&=\sum_{\boldsymbol{a}\notin \mathcal L}\Pr[M\cdot \boldsymbol{a}=0]+\sum_{\boldsymbol{a}\in \mathcal L}\Pr[M\cdot \boldsymbol{a}=0]\\
        &\leq 1+\frac{1+o(1)}{q^n}\cdot q^n+\binom{n}{n/\log^2n}q^{n/\log^2n+1}2^{-n}=2+o(1).
    \end{align*}
    This completes the proof.

{\bf Acknowledgement} The author is grateful to Vishesh Jain and Van Vu for helpful comments on the manuscript. The author is also grateful to Jozsi Balogh for suggesting some simplifications and improvements for Propositions \ref{proposition 1} and \ref{proposition 2}. 

\bibliographystyle{abbrv}
\bibliography{symmetric}
\end{document}